\documentclass[11pt,fleqn]{amsart}

\usepackage{amssymb,amsmath,amsthm}
\usepackage{graphicx}
\usepackage{color}
\usepackage{thmtools}
\usepackage{thm-restate}
\usepackage[shortlabels]{enumitem}
\usepackage{mathrsfs}
\usepackage[utf8]{inputenc}
\usepackage[T1]{fontenc}
\usepackage{dsfont}
\usepackage{soul}
\usepackage{quoting}

\usepackage{color}
\newtheorem{theorem}{Theorem}[section]
\newtheorem{lemma}[theorem]{Lemma}
\newtheorem{proposition}[theorem]{Proposition}

\newtheorem{question}[theorem]{Question}
\newtheorem{corollary}[theorem]{Corollary}

\theoremstyle{definition}

\theoremstyle{remark}
\newtheorem{remark}[theorem]{Remark}

\numberwithin{equation}{section}

\newcommand{\frakb}{\mathfrak{b}}
\newcommand{\frakc}{\mathfrak{c}}
\newcommand{\frakd}{\mathfrak{d}}

\newcommand{\fraks}{\mathfrak{s}}

\newcommand{\frakx}{\mathfrak{x}}

\newcommand{\frakz}{\mathfrak{z}}

\newcommand{\fraknik}{\mathfrak{nik}}

\newcommand{\frakgr}{\mathfrak{gr}}

\newcommand{\eps}{\varepsilon}

\newcommand{\B}{\mathbb{B}}

\newcommand{\R}{\mathbb{R}}

\newcommand{\G}{\mathbb{G}}
\newcommand{\PP}{\mathbb{P}}

\newcommand{\aA}{\mathcal{A}}
\newcommand{\bB}{\mathcal{B}}

\newcommand{\dD}{\mathcal{D}}

\newcommand{\fF}{\mathcal{F}}
\newcommand{\gG}{\mathcal{G}}
\newcommand{\hH}{\mathcal{H}}

\newcommand{\mM}{\mathcal{M}}

\newcommand{\pP}{\mathcal{P}}

\DeclareMathOperator{\cov}{cov}

\DeclareMathOperator{\id}{id}

\newcommand{\ol}{\overline}

\DeclareMathOperator{\ran}{ran}

\newcommand{\rstr}{\restriction}

\newcommand{\sm}{\setminus}
\newcommand{\sub}{\subseteq}
\DeclareMathOperator{\supp}{supp}

\newcommand{\wh}{\widehat}

\newcommand{\seq}[2]{\big\langle#1\colon\ #2\big\rangle}
\newcommand{\seqi}[1]{\big\langle#1\colon\ i\io\big\rangle}
\newcommand{\seqn}[1]{\big\langle#1\colon\ n\io\big\rangle}

\newcommand{\seqk}[1]{\big\langle#1\colon\ k\io\big\rangle}
\newcommand{\seql}[1]{\big\langle#1\colon\ l\io\big\rangle}

\newcommand{\clopen}[1]{\left[#1\right]}

\newcommand{\ctblsub}[1]{\left[#1\right]^\omega}
\newcommand{\finsub}[1]{\left[#1\right]^{<\omega}}

\newcommand{\iA}{\in\aA}

\newcommand{\io}{\in\omega}

\newcommand{\wo}{{\wp(\omega)}}
\newcommand{\oo}{\omega^\omega}

\newcommand{\cso}{\ctblsub{\omega}}

\newcommand{\ioo}{\in\oo}

\newcommand{\ooi}{(\oo)^\infty}
\newcommand{\iooi}{\in(\oo)^\infty}

\newcommand{\Cantor}{2^\omega}

\newcommand{\noproof}{\hfill$\Box$}

\newcommand{\forces}{\Vdash}

\renewcommand{\SS}{\mathbb{S}}

\begin{document}

\title[Adding a real]{Convergence of measures after adding a real}
\author[D.\ Sobota]{Damian Sobota}
\address{Kurt G\"{o}del Research Center, Institut f\"ur Mathematik, Universit\"at Wien, Kolingasse 14-16, 1090 Wien, Austria.}
\email{ein.damian.sobota@gmail.com}
\urladdr{www.logic.univie.ac.at/~{}dsobota}
\author[L. Zdomskyy]{Lyubomyr Zdomskyy}
\address{Institut f\"ur Diskrete Mathematik und Geometrie, Technische Universit\"at Wien, Wiedner Hauptstra\ss e 8-10/104, 1040 Wien, Austria.}
\email{lzdomsky@gmail.com}
\urladdr{dmg.tuwien.ac.at/zdomskyy}
\thanks{The authors were supported by the Austrian Science Fund FWF, Grants I 3709-N35 and I 4570-N35.}


\begin{abstract}
We prove that if $\mathcal{A}$ is an infinite Boolean algebra in the ground model $V$ and $\mathbb{P}$ is a notion of forcing adding any of the following reals: a Cohen real, an unsplit real, or a random real, then, in any $\mathbb{P}$-generic extension $V[G]$, $\mathcal{A}$ has neither the Nikodym property nor the Grothendieck property. A similar result is also proved for a dominating real and the Nikodym property.
\end{abstract}



\maketitle

\section{Introduction}

Let $\aA$ be a Boolean algebra. We say that $\aA$ has \textit{the Nikodym property}\footnote{For an equivalent definition of the Nikodym property in terms of bounded sequences of measures, see Lemma \ref{lemma:nik_equiv_forms}.} if every sequence $\seqn{\mu_n}$ of measures on $\aA$ which is pointwise null, i.e. $\mu_n(A)\to 0$ for every $A\iA$, is also weak* null, i.e. $\mu_n(f)\to 0$ for every continuous function $f\in C(St(\aA))$ on the Stone space of $\aA$, and that $\aA$ has \textit{the Grothendieck property} if every weak* null sequence $\seqn{\mu_n}$ of measures on $\aA$ is weakly null, i.e. $\mu_n(B)\to 0$ for every Borel $B\sub St(\aA)$ (see Section 2 for all the necessary terminology). Both of the notions have strong connections to functional analysis---the Nikodym property is closely related to the Uniform Boundedness Principle for locally convex spaces (see \cite{Sch82}), while the Grothendieck property is usually studied in a much more general sense in the context of dual Banach spaces (see \cite{Gro53}, \cite{Sch82} or \cite{Die73}). Nikodym \cite{Nik33} and Dieudonn\'e \cite{Die51} proved that all $\sigma$-complete Boolean algebras have the Nikodym property, while Grothendieck \cite{Gro53} showed that they have also the Grothendieck property. Consequently, e.g., the algebra $\wo$ of all subsets of $\omega$ has both of the properties. On the other hand, no infinite countable Boolean algebra (or, more generally, no Boolean algebra whose Stone space contains a non-trivial convergent sequence) can have the Nikodym property or the Grothendieck property.

Since the findings of Nikodym, Dieudonn\'e, and Grothendieck, many generalizations of the $\sigma$-completeness have been found which still give at least one of the properties, see e.g. \cite{See68}, \cite{Hay81}, \cite{Das81}, \cite{Mol81}, \cite{Sch82}, \cite{Fre84}, \cite{Aiz88}, \cite{Hay01}, \cite{DSsurv}. Unfortunately, none of those generalizations yields a necessary condition which a given Boolean algebra must satisfy in order to have the Nikodym property or the Grothendieck property. One of the reasons behind this is that, due to the result of Koszmider and Shelah \cite{KS13}, each of those generalizations implies also that an infinite Boolean algebra satisfying it contains an independent family of size continuum $\frakc$ and thus itself must be of cardinality at least $\frakc$. Brech \cite{Bre06} however showed that consistently there exists a Boolean algebra of cardinality $\omega_1$ having the Grothendieck property while at the same time $\frakc\ge\omega_2$. A similar result was also obtained by the first author \cite{DSNik} for the Nikodym property. Those two facts imply together that the quest for an algebraic or topological characterization of the Nikodym property or the Grothendieck property is much more demanding and requires using more sophisticated assumptions than mere existence of suprema or upper bounds of antichains in Boolean algebras.

Let us state the result of Brech \cite{Bre06} more precisely. She proved that if $\kappa$ is a cardinal number and $\SS(\kappa)$ is the side-by-side Sacks forcing adding simultaneously $\kappa$ many Sacks reals to the ground model $V$, then in any $\SS(\kappa)$-generic extension $V[G]$ the ground model Boolean algebra $\wo\cap V$ has the Grothendieck property (her argument works in fact for any infinite ground model $\sigma$-complete Boolean algebra, not only for $\wo$). In \cite{SZNik} we showed that a similar theorem may be obtained for the Nikodym property and in \cite{SZForExt} we generalized both of the results by proving that if $\PP$ is a proper notion of forcing satisfying the Laver property and preserving the reals non-meager, then in any $\PP$-generic extension $V[G]$ every ground model $\sigma$-complete Boolean algebra has both the Nikodym property and the Grothendieck property. Recall that the class of forcings satisfying the assumptions of the latter theorem contains such classical notions like the Sacks, side-by-side Sacks, Miller, or Silver(-like) forcing, as well as their countable support iterations (see \cite[Introduction]{SZForExt} for references).

In this paper we follow the path of research described in the previous paragraph and study the case of adding just \textit{one} real to the given model of set theory, however this time the results are  mostly negative. Our main theorem reads to wit as follows.

\begin{theorem}\label{theorem:main}
Let $\aA\in V$ be an infinite Boolean algebra. Let $\PP\in V$ be a notion of forcing adding one of the following reals:
\begin{itemize}
	\item a Cohen real,
	\item an unsplit real, or
	\item a random real.
\end{itemize}
Assume that $G$ is a $\PP$-generic filter over $V$. Then, in $V[G]$, $\aA$ has neither the Nikodym property nor the Grothendieck property.
\end{theorem}

We establish the above theorem in a series of partial results. First, in Theorems \ref{theorem:cohen} and \ref{theorem:unsplit} we prove that if $\PP$ adds a Cohen real or an unsplit real, then in any $\PP$-generic extension every infinite ground model Boolean algebra $\aA$ obtains a non-trivial convergent sequence in its Stone space $St(\aA)$, and, consequently, it can have neither the Nikodym property nor the Grothendieck property. Theorems \ref{theorem:cohen} and \ref{theorem:unsplit} have already been known to experts in the area (cf. e.g. Dow--Fremlin \cite[page 162]{DF07} and their reference to Koszmider \cite{Kos90}), however, it seems that their proofs have never been published anywhere. Our proof of Theorem \ref{theorem:cohen} may be seen as a forcing counterpart of the proof of Geschke's \cite[Theorem 2.1]{Ges06} which states that under Martin's axiom every infinite compact space of weight $<2^\omega$ contains a non-trivial convergent sequence or, more generally, that in ZFC every infinite compact space of weight strictly less than the covering number $\cov(\mM)$ of the meager ideal $\mM$ contains such a sequence. Geschke's argument is on the other hand a topological counterpart of Koppelberg's \cite[Proposition 5]{Kop77} asserting that under Martin's axiom every infinite Boolean algebra of cardinality $<2^\omega$ has countable cofinality. The argument for Theorem \ref{theorem:unsplit} is based on the idea presented in Booth \cite[Theorem 2]{Boo74} (see also \cite{vD84}) where it is showed that every infinite compact space of weight strictly less than the splitting number $\fraks$ is sequentially compact and thus contains a non-trivial convergent sequence.

The issue of adding random reals is more special. Recall that Dow and Fremlin \cite{DF07} first proved that adding any number of random reals to the ground model does not introduce non-trivial convergent sequences to the Stone spaces of $\sigma$-complete ground model Boolean algebras (or, more generally, to the Stone spaces of ground model Boolean algebras whose Stone spaces \textit{in} the ground model are F-spaces). Since not containing any non-trivial convergent sequences in the Stone space is not sufficient for an infinite Boolean algebra to have the Nikodym property or the Grothendieck property, the result of Dow and Fremlin does not say anything about the preservation of either of the properties by the random forcing. We address here this issue by proving in Theorem \ref{theorem:random} that if a forcing $\PP$ adds a random real, then for any infinite ground model Boolean algebra $\aA$ in every $\PP$-generic extension of the ground model there are sequences of finitely supported measures on the Stone space $St(\aA)$ which witness that $\aA$ has neither the Nikodym property nor the Grothendieck property. 

We also generalize partially the aforementioned result of Dow and Fremlin. Namely, we prove in Theorem \ref{theorem:df_gen} that for any ground model $\sigma$-complete Boolean algebra $\aA$ the random forcing does not add to its Stone space $St(\aA)$ any weak* null sequences of normalized measures whose supports consist of at most $M$ points, where $M\io$ is a fixed number. This result complements Theorem \ref{theorem:random}, at least in the case of $\sigma$-algebras---see Section \ref{sec:df} for more details. 

As examples of forcings adding a Cohen real one can name the Hechler forcing or finite support iterations of infinite length of non-trivial posets (see \cite[Example 0.2]{MG91}). The Mathias forcing is a typical example of a notion adding an unsplit real. Finally, random reals are added by, e.g., the amoeba forcing. 

\begin{corollary}\label{cor:main}
Let $\aA\in V$ be an infinite Boolean algebra. Let $\PP\in V$ be one of the following notions of forcing: Cohen, finite support iteration of infinite length of non-trivial posets, Hechler, Mathias, random, or amoeba. Assume that $G$ is a $\PP$-generic filter over $V$. Then, in $V[G]$, $\aA$ has neither the Nikodym property nor the Grothendieck property.
\end{corollary}

We also study the case of adding dominating reals---following the argument presented in \cite[Proposition 8.8]{DSNik} and based on the celebrated Josefson--Nissenzweig theorem from Banach space theory we prove in Section \ref{sec:dominating} that adding dominating reals kills the Nikodym property of all infinite ground model Boolean algebras.

\begin{restatable}{theorem}{dominating}
\label{theorem:main_dominating}
Let $\aA\in V$ be an infinite Boolean algebra. Let $\PP\in V$ be a notion of forcing adding a dominating real. Assume that $G$ is a $\PP$-generic filter over $V$. Then, in $V[G]$, $\aA$ does not have the Nikodym property.
\end{restatable}

The class of forcings adding a dominating real contains such notions as Hechler, Laver, or Mathias. Thus, in addition to Corollary \ref{cor:main}, we get the following result.

\begin{corollary}\label{cor:main2}
Let $\aA\in V$ be an infinite Boolean algebra. Let $\PP\in V$ be the Laver forcing. Assume that $G$ is a $\PP$-generic filter over $V$. Then, in $V[G]$, $\aA$ does not have the Nikodym property.
\end{corollary}

The case of the Laver forcing is particularly interesting as Dow \cite[Theorem 11]{Dow21} showed that adding a single Laver real does not introduce any non-trivial converging sequences in the Stone space of the ground model Boolean algebra $\wo\cap V$, yet, by Corollary \ref{cor:main2}, $\wo\cap V$ loses its Nikodym property. We do not know whether adding a Laver real (or, more generally, a dominating real) kills the Grothendieck property of ground model $\wo$ (or any other ground model Boolean algebra)---see Section \ref{sec:open}.

\section{Notations}

Our notations are standard---we follow the texts of Diestel \cite{Die84}, Kunen \cite{Kun80}, and Engelking \cite{Eng89}. We mention below only the most important issues.

$V$ always denotes the set-theoretic universe.

By $\omega$ we denote the first infinite countable ordinal number. If $A$ is a set, then by $\wp(A)$, $\ctblsub{A}$, and $\finsub{A}$ we denote the families of all subsets of $A$, all infinite countable subsets of $A$, and all finite subsets of $A$, respectively. $A^B$ denotes the family of all functions from a set $B$ to $A$. If $f$ is a function, then by $\ran(f)$ we denote its range. If $(L,\le)$ is a linear order and $f,g\in L^\omega$, then by writing $f\le g$ ($f\le^*g$) we mean that for all (but finitely many) $n\io$ we have $f(n)\le g(n)$. We similarly define the strict relations $<$ and $<^*$ on $L^\omega$. $\id_A$ denotes the identity function on $A$. If $B\sub A$, then by $\chi_B$ we denote the characteristic function of $B$ on $A$.

All topological spaces considered in this paper are assumed to be Tychonoff, that is, completely regular and Hausdorff. A subset of a topological space is \textit{perfect} if it is closed and contains no isolated points. A sequence $\seqn{x_n}$ in a topological space $X$ is \textit{non-trivial} if $x_n\neq x_m$ for every $n\neq m\io$ and, if the limit exists, $x_m\neq\lim_{n\to\infty}x_n$ for every $m\io$.

If $\aA$ is a Boolean algebra, then $St(\aA)$ denotes its Stone space (i.e. the space of all ultrafilters on $\aA$) with the usual topology which makes it a totally disconnected compact Hausdorff space. Recall that $\aA$ is isomorphic to the algebra of clopen subsets of $St(\aA)$. For every element $A\in\aA$ by $[A]_\aA$ we denote the clopen subset of $St(\aA)$ corresponding to $A$.

If we say that \textit{$\mu$ is a measure on a Boolean algebra $\aA$}, then we mean that $\mu$ is a signed  finitely additive function from $\aA$ to $\R$ with bounded total variation, that is, the following holds:
\[\|\mu\|=\sup\big\{|\mu(A)|+|\mu(B)|\colon\ A,B\iA, A\wedge B=0\big\}<\infty.\]
When we say that $\mu$ \textit{is a measure on a compact Hausdorff space $K$}, then we mean that $\mu$ is a signed $\sigma$-additive Radon measure defined on the Borel $\sigma$-algebra $Bor(K)$ of $K$---it follows automatically that $\mu$ has bounded total variation, that is:
\[\|\mu\|=\sup\big\{|\mu(A)|+|\mu(B)|\colon\ A,B\in Bor(K), A\cap B=0\big\}<\infty.\]
Recall that if we identify a given Boolean algebra $\aA$ with the subalgebra of clopen subsets of the Borel $\sigma$-field $Bor(St(\aA))$, then every measure $\mu$ on $\aA$ extends uniquely to a measure $\wh{\mu}$ on $St(\aA)$---we will usually omit $\hat{}$ and write simply $\mu$, too.

Let $K$ be a compact space. For a measure $\mu$ on $K$ and a $\mu$-measurable function $f\colon K\to\R$ we write $\mu(f)$ to denote $\int_Kfd\mu$. By $C(K)$ we denote the Banach space of all continuous real-valued functions on $K$ endowed with the supremum norm. Recall that by the Riesz representation theorem the dual space $C(K)^*$ is isometrically isomorphic to the Banach space $M(K)$ of all Radon measures on $K$ endowed with the total variation norm---$M(K)$ acts on $C(K)$ by the formula $\langle f,\mu\rangle=\mu(f)$.

Let $\seqn{\mu_n}$ be a sequence of measures on a Boolean algebra $\aA$. If $\lim_{n\to\infty}\mu_n(A)=0$ for every $A\iA$, then we say that $\seqn{\mu_n}$ is \textit{pointwise null}; if $\lim_{n\to\infty}\mu_n(f)=0$ for every $f\in C(St(\aA))$, then it is \textit{weak* null}; and if $\lim_{n\to\infty}\mu_n(B)=0$ for every $B\in Bor(St(\aA))$, then it is \textit{weakly null} (cf. \cite[Theorem 11, page 90]{Die84}). Additionally, we say that $\seqn{\mu_n}$ is \textit{pointwise bounded} if $\sup_{n\io}\big|\mu_n(A)\big|<\infty$ for every $A\iA$, and that it is \textit{uniformly bounded} if $\sup_{n\io}\big\|\mu_n\big\|<\infty$.

\section{Adding a convergent sequence}

In this section we prove that adding a Cohen real (Theorem \ref{theorem:cohen}) or an unsplit real (Theorem \ref{theorem:unsplit}) to the ground model produces a non-trivial convergent sequence in the Stone space of every infinite ground model Boolean algebra. Notice that using the methods described in \cite[page 162]{DF07} one can generalize those results to any infinite ground model compact space $K$.

As we mentioned in Introduction, both of the theorems have been already known to some experts, but it seems that their proofs have never been published anywhere.

\subsection{Cohen reals}

Let us first recall the definition of a Cohen real. Let $\PP\in V$ be a notion of forcing and $G$ a $\PP$-generic filter over $V$. Then, $x\in\Cantor\cap V[G]$ is \textit{a Cohen real over $V$} if for every  dense subset $D\sub Fn(\omega,2)$ such that $D\in V$ we have $D\cap\big\{p\in Fn(\omega,2)\colon\ p\sub x\big\}\neq\emptyset$. Here $Fn(\omega,2)$ is the family of all finite partial functions from $\omega$ to $2$, ordered by the reverse inclusion.

We will need the following folklore lemma.

\begin{lemma}\label{lemma:scattered_conv_seq}
If $K$ is an infinite scattered compact Hausdorff space, then $K$ contains a non-trivial convergent sequence.
\end{lemma}
\begin{proof}
Since $K$ is scattered and infinite, there is a countable subset $A$ of $K$ such that every $x\in A$ is isolated in $K$. $A$ must be discrete and open in $K$. Since $K$ is compact, the boundary $\partial A$ is non-empty and thus must contain an isolated point $x$ (in $\partial A$). The sets $\{x\}$ and $(\partial A)\sm\{x\}$ are closed subsets of $K$, so there are disjoint open sets $V$ and $W$ such that $\{x\}\sub V$ and $(\partial A)\sm\{x\}\sub W$. Note that $\ol{V\cap A}=(V\cap A)\cup\{x\}$, so $\ol{V\cap A}$ is a one-point compactification of $V\cap A$. Enumerate $V\cap A=\{x_n:\ n\io\}$; then $x_n\to x$.
\end{proof}

Now, we are in the position to prove the main theorem of this section.

\begin{theorem}\label{theorem:cohen}
Let $\PP\in V$ be a notion of forcing adding a Cohen real and $\aA\in V$ an infinite Boolean algebra. Then, for every $\PP$-generic filter $G$ over $V$ the Stone space $\big(St(\aA)\big)^{V[G]}$ contains a non-trivial convergent sequence.
\end{theorem}
\begin{proof}
We have two cases:

1) In $V$, the Stone space $St(\aA)$ of $\aA$ is scattered---by Lemma \ref{lemma:scattered_conv_seq} there is a non-trivial convergent sequence in $St(\aA)$. Of course, this sequence will also be convergent in the Stone space of $\aA$ in any $\PP$-generic extension $V[G]$.

2) In $V$, the Stone space $St(\aA)$ is not scattered. Hence, there is a closed subset $L$ of $St(\aA)$ and a continuous surjection $f\colon L\to\Cantor$. By the Kuratowski--Zorn lemma, we may assume that $f$ is irreducible and hence that $L$ is perfect. The family
\[\pP=\big\{f^{-1}[U]\colon\ U\neq\emptyset\text{ is a clopen in }\Cantor\big\}\]
is a countable $\pi$-base of $L$ (partially ordered by the reverse inclusion $\supseteq$). Indeed, given any non-empty open set $W\sub L$, note that $f[L\sm W]\neq\Cantor$ by the irreducibility of $f$, so for any clopen $U\sub\Cantor\sm f[L\sm W]$ we have $f^{-1}[U]\sub W$.

Let $\bB$ be the Boolean algebra of clopen subsets of $L$. Of course, $\pP\sub\bB$. By the Stone duality, $\bB$ is a homomorphic image of $\aA$. For every $U\in\bB$ put:
\[D_U=\big\{P\in\pP\colon\ P\sub U\text{ or }P\sub L\sm U\big\}.\]
Trivially, each $D_U\in V$ and is dense in the poset $(\pP,\supseteq)$.

Fix now a $\PP$-generic filter $G$ over $V$ and let us work in $V[G]$. By the assumption, there is a Cohen real $c\in\Cantor$ over $V$. The family
\[\gG=\big\{f^{-1}\big[[c\rstr n]^V\big]\colon n\io\big\}\]
is a $\pP$-generic filter over $V$, so, in particular, $\gG$ meets every $D_U$ (as $D_U\in V$). Let $x\in St(\bB)$ be the ultrafilter with the base $\gG$. 
Since the ground model (perfect) set $L$ had no isolated points (in $V$) and it is dense in $St(\bB)$, $x$ is not isolated in $St(\bB)$. Thus, we proved that $St(\bB)$ is a perfect set containing a $\G_\delta$-point. In particular, $St(\bB)$ contains a non-trivial convergent sequence.

In $V[G]$, $\bB$ is still a homomorphic image of $\aA$, hence $St(\bB)$ is homeomorphic to a closed subset of $St(\aA)$. By the previous paragraph, $St(\aA)$ contains a non-trivial convergent sequence (in $V[G]$).
%
%
\end{proof}

The next corollary follows from the proof of Theorem \ref{theorem:cohen}. Recall that a point $x$ in a topological space $X$ is \textit{a $\G_\delta$-point} if the singleton $\{x\}$ is the intersection of a countable family of open subsets of $X$.

\begin{corollary}
Let $\PP\in V$ be a notion of forcing adding a Cohen real and $\aA\in V$ an infinite Boolean algebra such that its Stone space $St(\aA)$ is not scattered. Then, for every $\PP$-generic filter $G$ over $V$ the Stone space $\big(St(\aA)\big)^{V[G]}$ contains a perfect subset $L$ and a point $x\in L$ which is a $\G_\delta$-point in $L$.\noproof
\end{corollary}

\subsection{Unsplit reals}

Let $\PP\in V$ be a notion of forcing and $G$ a $\PP$-generic filter over $V$. We say that a real $U\in \wo\cap V[G]$ is \textit{unsplit} if for every $A\in\wo\cap V$ the set $U\cap A$ is finite or the set $U\sm A$ is finite. 

The proof of the following theorem follows the idea of Booth \cite[Theorem 2]{Boo74} (see also \cite{vD84}).

\begin{theorem}\label{theorem:unsplit}
Let $\PP\in V$ be a notion of forcing adding an unsplit real and $\aA\in V$ an infinite Boolean algebra. Then, for every $\PP$-generic filter $G$ over $V$ the Stone space $\big(St(\aA)\big)^{V[G]}$ contains a non-trivial convergent sequence.
\end{theorem}
\begin{proof}
We work first in $V$. Let $A\sub St(\aA)$ be an infinite countable set. Put:
\[\dD=\big\{A\cap[B]_\aA\colon\ B\in\aA,\ |A\cap[B]_\aA|=\omega\big\}.\]
Obviously, $\dD\sub\ctblsub{A}$.

Fix a $\PP$-generic filter $G$ over $V$ and let us now work in $V[G]$. By the assumption, there exists $U\sub\ctblsub{A}$ which is unsplit by $\big(\ctblsub{A}\big)^V$. It follows that for every $D\in\dD$ the set $U\cap D$ is finite or the set $U\sm D$ is finite. Since $St(\aA)$ is compact, there is a limit point $x$ of $U$ in $St(\aA)$. Enumerate $U=\{x_n\colon\ n\io\}$. We claim that the sequence $\seqn{x_n}$ converges to $x$. Indeed, let $B\in\aA$ be such that $x\in[B]_\aA$. Since $|U\cap[B]_\aA|=\omega$ and $U\in\ctblsub{A}$, we have that $|A\cap[B]_\aA|=\omega$. Note that the set $A\cap\clopen{B}_\aA$ is in $V$, 
so we get that $A\cap[B]_\aA\in\dD$, which implies that the set $U\sm[B]_\aA=U\sm(A\cap[B]_\aA)$ is finite.
\end{proof}

\section{Destroying the Nikodym property or the Grothendieck property\label{sec:destroying}}

In this section we provide two negative results. Namely, in Theorem \ref{theorem:random} we prove that adding a random real causes that no ground model Boolean algebra has the Nikodym property or the Grothendieck property, and in Theorem \ref{theorem:main_dominating} we show that after adding a dominating real no ground model Boolean algebra has the Nikodym property. We do not know whether adding dominating reals kills the Grothendieck property---see Questions \ref{ques:laver_gr} and \ref{ques:dom_gr}.

We start the section recalling several auxiliary facts---the first lemma provides an alternative definition for the Nikodym property (in fact, the one more commonly used in the literature, however lacking the apparent similarity to the definition of the Grothendieck property).

\begin{lemma}\label{lemma:nik_equiv_forms}
Let $\aA$ be a Boolean algebra. The following two conditions are equivalent:
\begin{enumerate}
	\item every pointwise null sequence of measures on $\aA$ is weak* null;
	\item every pointwise bounded sequence of measures on $\aA$ is uniformly bounded.
\end{enumerate}
\end{lemma}
\begin{proof}
Assume (1) and suppose that there exists a sequence $\seqn{\mu_n}$ of measures on $\aA$ which is pointwise bounded but not uniformly bounded. By going to the subsequence, we may assume that $\big\|\mu_n\big\|>n$ for every $n\io$. For each $n\io$ define the measure $\nu_n$ on $\aA$ as follows:
\[\nu_n=\mu_n\big/\sqrt{\big\|\mu_n\big\|}.\]
It follows that $\big\|\nu_n\big\|=\sqrt{\big\|\mu_n\big\|}>\sqrt{n}$. 
On the other hand, for every $A\iA$ we have:
\[\big|\nu_n(A)\big|=\big|\mu_n(A)\big|\big/\sqrt{\big\|\mu_n\big\|},\]
which converges to $0$ as $n\to\infty$ (because $\sup_{n\io}\big|\mu_n(A)\big|<\infty$), which contradicts (1) as weak* null sequences are always uniformly bounded (by the virtue of the Banach--Steinhaus theorem). Hence, (2) holds.

Assume now (2) and let $\seqn{\mu_n}$ be a pointwise null sequence of measures on $\aA$. It follows immediately that $\seqn{\mu_n}$ is pointwise bounded, hence, by (2), it is uniformly bounded. Let $M>0$ be such that $\sup_{n\io}\big\|\mu_n\big\|<M$. Fix $f\in C(St(\aA))$ and let $\eps>0$. There are finite sequences $A_1,\ldots,A_k\iA$ and $\alpha_1,\ldots,\alpha_k\in\R$ such that
\[\Big\|f-\sum_{i=1}^k\alpha_i\cdot\chi_{\clopen{A_i}_\aA}\Big\|<\eps/(2M).\]
Since $\seqn{\mu_n}$ is pointwise null, there is $N\io$ such that for every $n>N$ we have:
\[\sum_{i=1}^k\big|\alpha_i\big|\cdot\big|\mu_n\big(A_i\big)\big|<\eps/2.\]
Thus, for every $n>N$ it holds:
\[\big|\mu_n(f)\big|<\big|\mu_n\Big(f-\sum_{i=1}^k\alpha_i\cdot\chi_{\clopen{A_i}_\aA}\Big)\big|+\big|\mu_n\Big(\sum_{i=1}^k\alpha_i\cdot\chi_{\clopen{A_i}_\aA}\Big)\big|\le\]
\[\le\big\|\mu_n\big\|\cdot\Big\|f-\sum_{i=1}^k\alpha_i\cdot\chi_{\clopen{A_i}_\aA}\Big\|+\sum_{i=1}^k\big|\alpha_i\big|\cdot\big|\mu_n\big(A_i\big)\big|<\eps.\]
It follows that $\mu_n(f)\to0$ as $n\to\infty$, which proves that $\seqn{\mu_n}$ is weak* null. Consequently, (1) holds.
\end{proof}

From the proof of implication (2)$\Rightarrow$(1) we immediately get the following corollary.

\begin{corollary}\label{cor:pnub_wsn}
Let $\aA$ be a Boolean algebra. If $\seqn{\mu_n}$ is a pointwise null uniformly bounded sequence of measures on $\aA$, then $\seqn{\mu_n}$ is weak* null.\noproof
\end{corollary}

If $X$ is a topological space and $x\in X$, then by $\delta_x$ we denote the Borel one-point measure on $X$ concentrated at $x$. Recall that a measure $\mu$ on a compact space $K$ (a Boolean algebra $\aA$) is \textit{finitely supported} or has \textit{finite support} if there exist finite sequences $x_1,\ldots,x_n$ of pairwise distinct points in $K$ (in $St(\aA)$) and $\alpha_1,\ldots,\alpha_n\in\R$ such that $\mu=\sum_{i=1}^n\alpha_i\delta_{x_i}$. The set $\big\{x_1,\ldots,x_n\big\}$ is called \textit{the support} of $\mu$ and denoted by $\supp(\mu)$.

We will need the following simple lemma.

\begin{lemma}\label{lemma:fs_no_nik_no_gr}
Let $\aA$ be a Boolean algebra. If there exists a sequence $\seqn{\mu_n}$ of finitely supported measures on $\aA$ which is pointwise null but not uniformly bounded, then $\aA$ has neither the Nikodym property nor the Grothendieck property.
\end{lemma}
\begin{proof}
$\seqn{\mu_n}$ directly witnesses the lack of the Nikodym property. Consider the sequence $\seqn{\nu_n}$ defined as $\nu_n=\mu_n/\big\|\mu_n\big\|$ for every $n\io$. Since it is pointwise null, too, and uniformly bounded, by Corollary \ref{cor:pnub_wsn} it is weak* null. Set $S=\bigcup_{n\io}\supp\big(\nu_n\big)$ and note that the Banach space $\ell_1(S)$ of all absolutely summable sequences on the set $S$ is a closed linear subspace of the dual space $C(St(\aA))^*$ containing every $\nu_n$. Since $\ell_1(S)$ has the Schur property (meaning that the weak convergence of sequences implies their norm convergence), the sequence $\seqn{\nu_n}$ cannot be weakly null, as $\big\|\nu_n\big\|=1$ for every $n\io$. In particular, $\aA$ does not have the Grothendieck property.
\end{proof}

\subsection{Random reals. Destroying the Nikodym and Grothendieck properties\label{sec:random}}

In order to prove Theorem \ref{theorem:random}, we need to recall some basic facts concerning the binomial distributions. Let $(\Omega,\Sigma,\Pr)$ be a probability space. Given $p\in (0,1)$, for every $i\io$ let $X_i$ be a random variable taking only two values: $0$ and $1$, and such that the following two equalities hold:
\[\Pr\Big(\big\{t\in\Omega: X_i(t)=1\big\}\Big)=\Pr\Big(X_i^{-1}(1)\Big)=p,\]
and
\[\Pr\Big(\big\{t\in\Omega: X_i(t)=0\big\}\Big)=\Pr\Big(X_i^{-1}(0)\Big)=1-p.\]
Assume additionally that the sequence $\seqi{X_i}$ is \textit{independent}, that is, for every $n>0$ and $s\in 2^n$ we have
\[\Pr\Big(\big\{t\in\Omega\colon\ X_i(t)=s(i)\text{ for every }i<n\big\}\Big)=\Pr\Big(\bigcap_{i<n}X_i^{-1}\big(s(i)\big)\Big)=\prod_{i<n} p_i,\]
where $p_i=p$ if $s(i)=1$, and $p_i=1-p$ otherwise. The following classical fact is crucial for our proof of Theorem \ref{theorem:random}; for its proof see e.g. \cite[Section 1.3]{Bol01}. Recall that $\exp(x)=e^x$ for $x\in\R$.

\begin{theorem}\label{theorem:close_to_center}
Suppose that $p\in(0,1/2]$, $m\io$, and $\eps\in(0,1/12]$ are such that $\eps p(1-p)m\ge 12$. Then,
\[\Pr\Big(\big\{t\in\Omega\colon\ \big|\sum_{i<m}X_i(t)-pm\big|\ge\eps pm\big\}\Big) \le (\eps^2 pm)^{-1/2}\cdot\exp\big(-\eps^2 pm/3\big).\]\noproof
\end{theorem}

In what follows we fix $p=1/2$. Put $\Omega=\Cantor$ and let $\Sigma$ denote the standard Borel $\sigma$-field on $\Omega$ and $\lambda$ the standard product measure on $\Omega$. We will now work in the probability space $(\Omega,\Sigma,\lambda)$. For every $i\io$ and $x\in\Omega$ set $X_i(x)=x(i)$, i.e., the function $X_i$ is simply the projection onto the $i$-th coordinate. Obviously, the sequence $\seqi{X_i}$ of random variables is as described in the paragraph before Theorem \ref{theorem:close_to_center}.

\begin{lemma}\label{lem:prob_borelcant}
For every $n\io$ set $I_n=\big\{2^n+1, 2^n+2,\ldots, 2^{n+1}\big\}$. Suppose that for some infinite $J\subset\omega$ and for every $n\in J$ there is a subset $Y_n\sub I_n$ such that $\eta=\inf\big\{\eta_n\colon\ n\in J\big\}>0$, where $\eta_n=\big|Y_n\big|/2^n$ for each $n\io$. For each $n\in J$ let $m_n=\big|Y_n\big|$ ($=\eta_n 2^n$) and $\eps_n=\sqrt{n/2^n}$, and assume that $\eps_n\le 1/12$ and $\eps_n p^2 m_n=\frac{1}{4}\eta_n\sqrt{n 2^n}\ge 12$. For every $n\in J$ put:
\[A_n=\Big\{x\in\Cantor\colon\ \Big|\sum_{i\in Y_n}x(i)-\frac{1}{2}\cdot\eta_n2^n\Big|\ge\frac{1}{2}\cdot\eta_n\sqrt{n 2^n}\Big\}.\]
Then,
\[\tag{$\dagger$}\lambda\Big(\bigcup_{n\in J}\bigcap_{\substack{k\in J\\k\ge n}}A_k^c\Big)=\lambda\big(\big\{x\in\Cantor\colon\ x\not\in A_n\text{ for almost all }n\in J\big\}\big)=1.\]
\end{lemma}
\begin{proof}
Applying Theorem \ref{theorem:close_to_center} (for $m=m_n$ and $\eps=\eps_n$), for every $n\in J$ we get:
\begin{align*}
\lambda\Big(\Big\{x\in\Cantor\colon\ \Big|\sum_{i\in Y_n}x(i)-\frac{1}{2}\cdot\eta_n 2^n\Big|&\ge\sqrt{n/2^n}\cdot\frac{1}{2}\cdot\eta_n2^n\Big\}\Big)\le\\  &\le\Big(\frac{n}{2^n}\cdot\frac{1}{2} \cdot \eta_n2^n\Big)^{-1/2}\cdot\exp\Big(-\frac{n}{2^n}\cdot\frac{1}{2}\cdot\eta_n 2^n\cdot\frac{1}{3}\Big),
\end{align*}
which after simplification reduces to:
\[\tag{$*$}\lambda\big(A_n\big)=\lambda\Big(\Big\{x\in\Cantor\colon\Big|\sum_{i\in Y_n}x(i)-\frac{1}{2}\cdot \eta_n2^n\Big|\geq\frac{1}{2}\cdot\eta_n\sqrt{n 2^n}\Big\}\Big)\le\sqrt{\frac{2}{n\eta_n}}\cdot\exp\big(-n\eta_n/6\big).\]
Observe that ($*$) actually implies that $\sum_{n\in J}\lambda(A_n)<\infty$ (because $\eta>0$), and hence by the Borel--Cantelli lemma we get ($\dagger$):
\[\lambda\Big(\bigcup_{n\in J}\bigcap_{\substack{k\in J\\k\ge n}}A_k^c\Big)=\lambda\big(\big\{x\in\Cantor\colon\ x\not\in A_n\text{ for almost all }n\in J\big\}\big)=1.\]
\end{proof}

We are now in the position to present the proof of the main theorem of this section. We will use the following definition of a random real: Given a (forcing) extension $V'$ of $V$, a real $r\in\Cantor$ is \textit{a random real over $V$} if for every Borel subset $B$ of $\Cantor$, coded in $V$ and such that $\big(\lambda(B)=0\big)^V$, the real $r$ does not belong to the interpretation of $B$ in $V'$. (We will abuse the notation and denote this interpretation by $B$, too.)

\begin{theorem}\label{theorem:random}
Let $\PP\in V$ be a notion of forcing adding a random real and $\aA\in V$ an infinite Boolean algebra. Assume that $G$ is a $\PP$-generic filter over $V$. Then, in $V[G]$, $\aA$ has neither the Nikodym property nor the Grothendieck property.
\end{theorem}
\begin{proof}
In $V$, let $\seqi{x_i}$ be a sequence of ultrafilters in $St(\aA)$ such that $x_i\neq x_j$ for $i\neq j\io$.

From now on we work exclusively in $V[G]$. 
Let $\varphi\colon\omega\to St(\aA)$ be such that $\varphi(i)=x_i$ for every $i\io$. Let $r\in\Cantor\cap V[G]$ be a random real over $V$. Set $\omega_+=\omega\sm\{0\}$. For every $n\io_+$ consider the measure $\mu_n$ on $\aA$ defined as follows:
\[\mu_n(A)=\alpha_n\cdot\sum_{i\in I_n}(-1)^{r(i)+1}\cdot\delta_{x_i}\big(\clopen{A}_\aA\big),\]
where $\alpha_n=1/\big(n\sqrt{2^n}\big)$ and $I_n=\big\{2^n+1,2^n+2,\ldots,2^{n+1}\big\}$. It follows that $\mu_n$ is finitely supported, $\supp\big(\mu_n\big)=\varphi\big[I_n\big]$, and
\[\big\|\mu_n\big\|=\alpha_n\cdot 2^n=\sqrt{2^n}/{n},\]
so $\lim_{n\to\infty}\big\|\mu_n\big\|=\infty$.

We claim that $\seqn{\mu_n}$ is pointwise null. Let us fix $A\iA$ and for every $n\io_+$ set
\[Y_n=\big\{i\in I_n\colon\ A\in x_i\big\}.\]
Of course, $Y_n\in V$. 
Put:
\[J=\big\{n\io_+\colon\ \big|Y_n\big|\big/2^n\ge1/2\big\}\quad\text{and}\quad J^c=\omega_+\sm J=\big\{n\io_+\colon\ \big|Y_n\big|\big/2^n<1/2\big\}.\]
Again, $J,J^c\in V$.

Assume first that $J$ is infinite. We will prove that $\mu_n(A)\to 0$ as $n\to\infty$, $n\in J$. For every $n\in J$ set also $\eta_n=\big|Y_n\big|/2^n$ and let $A_n$ be the clopen subset of $\Cantor$ such as defined in Lemma \ref{lem:prob_borelcant}. 
By the definition of $J$, we get that
\[\eta=\inf\big\{\eta_n\colon\ n\in J\big\}\ge1/2>0,\]
hence equation ($\dagger$) of Lemma \ref{lem:prob_borelcant} together with the definition of a random real imply that $r\not\in A_n$ for all but finitely many $n\in J$, which means that
\[\Big|\sum_{i\in Y_n}r(i)-\frac{1}{2}\cdot\eta_n 2^n\Big|<\frac{1}{2}\cdot\eta_n\sqrt{n 2^n}\]
for all but finitely many $n\in J$, and thus there is $n_0\io_+$ such that for all $n\in J$, $n\ge n_0$ we have (note that $\eta_n\le1$):
\[\Big|\sum_{i\in Y_n}r(i)-\big|Y_n\big|/2\Big|<\frac{1}{2}\cdot \sqrt{n 2^n},\]
which in turns implies that for all $n\in J$, $n\ge n_0$, and $s\in\{0,1\}$ it holds:
\[\Big|\big|\big\{i\in Y_n\colon\ r(i)=s\big\}\big|-\big|Y_n\big|/2\Big|<\frac{1}{2}\cdot \sqrt{n 2^n}.\]
(Just note that the values on the left hand sides of the latter two inequalities are the same.)
As a result, for every $n\in J$, $n\ge n_0$, we have:
\[\big|\mu_n(A)\big|=\big|\mu_n\big(\varphi\big[Y_n\big]\big)\big|=\big|\alpha_n\cdot\sum_{i\in Y_n}(-1)^{r(i)+1}\big|=\]
\[=\alpha_n\cdot\Big|\big|\big\{i\in Y_n\colon\ r(i)=1\big\}\big|-\big|\big\{i\in Y_n\colon\ r(i)=0\big\}\big|\Big|\le\]
\[\le\alpha_n\cdot\Big(\Big|\big|\big\{i\in Y_n\colon\ r(i)=1\big\}\big|-\big|Y_n\big|/2\Big|+\Big|\big|\big\{i\in Y_n\colon\ r(i)=0\big\}\big|-\big|Y_n\big|/2\Big|\Big)<\]
\[<\alpha_n\cdot\Big(\frac{1}{2}\cdot\sqrt{n 2^n}+\frac{1}{2}\cdot \sqrt{n 2^n}\Big)=\alpha_n\cdot\sqrt{n 2^n}=\frac{1}{n\sqrt{2^n}}\cdot\sqrt{n 2^n}=\frac{1}{\sqrt{n}},\]
which yields that
\[\lim_{\substack{n\to\infty\\n\in J}}\mu_n(A)=0.\]

If $J^c$ is finite, then we are immediately done, so assume that it is infinite. Notice that since for the unit element $1_\aA$ of the Boolean algebra $\aA$ and  all $i\io$ we have $1_\aA\in x_i$, exactly the same reasoning as above shows that $\lim_{n\to\infty}\mu_n\big(1_\aA\big)=0$, so in particular we have:
\[\lim_{\substack{n\to\infty\\n\in J^c}}\mu_n\big(1_\aA\big)=0.\]
For each $n\io_+$ define the set $Y_n'$ in $V$ similarly as $Y_n$:
\[Y_n'=\big\{i\in I_n\colon\ 1_\aA\sm A\in x_i\big\},\]
and put:
\[J'=\big\{n\io_+\colon\ \big|Y_n'\big|\big/2^n\ge1/2\big\}.\]
Since $Y_n'=I_n\sm Y_n$, we have:
\[J^c=\big\{n\io_+\colon\ \big|Y_n\big|\big/2^n<1/2\big\}=\big\{n\io_+\colon\ \big|I_n\sm Y_n\big|\big/2^n>1/2\big\}\sub\]
\[\sub\big\{n\io_+\colon\ \big|I_n\sm Y_n\big|\big/2^n\ge1/2\big\}=\big\{n\io_+\colon\ \big|Y_n'\big|\big/2^n\ge1/2\big\}=J',\]
so $J'$ is infinite. Using again the same argument as above, we show that
\[\lim_{\substack{n\to\infty\\n\in J'}}\mu_n\big(1_\aA\sm A\big)=0,\]
so in particular we get that
\[\lim_{\substack{n\to\infty\\n\in J^c}}\mu_n\big(1_\aA\sm A\big)=0.\]
Finally, we have:
\[\lim_{\substack{n\to\infty\\n\in J^c}}\mu_n(A)=\lim_{\substack{n\to\infty\\n\in J^c}}\mu_n\big(1_\aA\big)-\lim_{\substack{n\to\infty\\n\in J^c}}\mu_n\big(1_\aA\sm A\big)=0,\]
which ultimately implies that
\[\lim_{n\to\infty}\mu_n(A)=0.\]

We have just showed that the sequence $\seq{\mu_n}{n\io_+}$ of finitely supported measures on $\aA$ is pointwise null but not uniformly bounded, so, by Lemma \ref{lemma:fs_no_nik_no_gr}, $\aA$ has neither the Nikodym property nor the Grothendieck property. The proof is thus finished.
\end{proof}

Note that, normalizing measures $\mu_n$ from the above proof (that is, considering the measures $\mu_n/\big\|\mu_n\big\|$), by Corollary \ref{cor:pnub_wsn} we obtain the following result.

\begin{corollary}\label{cor:random_jn}
Let $\PP\in V$ be a notion of forcing adding a random real and $\aA\in V$ an infinite Boolean algebra. Assume that $G$ is a $\PP$-generic filter over $V$. Then, in $V[G]$, $St(\aA)$ carries a weak* null sequence $\seqn{\mu_n}$ of finitely supported measures with pairwise disjoint supports such that $\big\|\mu_n\big\|=1$ and $\big|\supp\big(\mu_n\big)\big|=2^n$ for every $n\io$.\noproof
\end{corollary}

\subsection{Random reals. Generalization of Dow--Fremlin's result\label{sec:df}}

Let $\kappa$ be an infinite cardinal number. By $\mu_\kappa$ denote the standard product probability measure on the space $2^\kappa$ and let $\B(\kappa)=Bor\big(2^\kappa\big)\big/\big\{A\in Bor\big(2^\kappa\big)\colon\ \mu_\kappa(A)=0\big\}$ be its measure algebra. $\B(\kappa)$ is a well-known $\oo$-bounding poset adding $\kappa$ many random reals (see \cite[Section 3.1]{BJ95})\footnote{Or,  formally, $\B(\kappa)\sm\{0\}$, but for simplicity we will keep writing just $\B(\kappa)$.}. (A forcing $\PP$ is \textit{$\oo$-bounding} if for every $\PP$-generic filter $G$ over $V$ and a function $f\in\oo\cap V[G]$ there is a function $g\in\oo\cap V$ such that $f(n)<g(n)$ for every $n\io$.)

Recall again that Dow and Fremlin \cite{DF07} proved that forcing with $\B(\kappa)$ does not introduce non-trivial convergent sequences to the Stone spaces of $\sigma$-complete ground model Boolean algebras. In this subsection we will generalize their result in the following way.

\begin{theorem}\label{theorem:df_gen}
Let $\aA\in V$ be an infinite $\sigma$-complete Boolean algebra. Assume that $G$ is a $\B(\kappa)$-generic filter over $V$. Then, in $V[G]$, $St(\aA)$ does not carry any weak* null sequence $\seqn{\mu_n}$ of finitely supported measures such that $\big\|\mu_n\big\|=1$ for every $n\io$ and for which there exists $M\io$ such that $\big|\supp\big(\mu_n\big)\big|\le M$ for every $n\io$. 
\end{theorem}

The theorem really generalizes the result of Dow and Fremlin, since if a compact space $K$ contains a non-trivial convergent sequence $\seqn{x_n}$, then the sequence $\seqn{\mu_n}$ of measures defined as $\mu_n=\frac{1}{2}\big(\delta_{x_{2n}}-\delta_{x_{2n+1}}\big)$ is weak* null and such that $\big\|\mu_n\big\|=1$ and $\big|\supp\big(\mu_n\big)\big|=2$ for every $n\io$. Note however that the existence of a weak* null sequence $\seqn{\nu_n}$ of measures on a totally disconnected compact space $K$ such that $\big|\supp\big(\nu_n\big)\big|=2$ for every $n\io$ does not imply the existence of non-trivial convergent sequences in $K$ (cf. \cite[Example 4.10]{Sch82}).

Let us stress that Corollary \ref{cor:random_jn} and Theorem \ref{theorem:df_gen} are complementary: the corollary states that the existence of a random real yields the existence of a weak* null sequence $\seqn{\mu_n}$ of finitely supported normalized measures on the Stone space of a given infinite ground model Boolean algebra such that $\lim_{n\to\infty}\big|\supp\big(\mu_n\big)\big|=\infty$, while the theorem asserts that, in the case of $\sigma$-complete Boolean algebras, this is optimal---we cannot get any weak* null sequence $\seqn{\mu_n}$ of normalized measures such that $\sup_{n\io}\big|\supp\big(\mu_n\big)\big|<\infty$.

In order to prove Theorem \ref{theorem:df_gen}, we first need to recall several auxiliary results. The first one implies that in fact we only need to deal with the case of $M=2$. 

\begin{lemma}\label{lemma:reduction_m_2}
Let $\aA$ be an infinite Boolean algebra. If there are a weak* null sequence $\seqn{\mu_n}$ of finitely supported measures on $St(\aA)$ and a number $M\io$ such that $\big\|\mu_n\big\|=1$ and $\big|\supp\big(\mu_n\big)\big|\le M$ for every $n\io$, then $M\ge2$ and there is a weak* null sequence $\seqn{\nu_n}$ of finitely supported measures on $St(\aA)$ such that $\big\|\nu_n\big\|=1$ and $\big|\supp\big(\nu_n\big)\big|=2$.
\end{lemma}
\begin{proof}
Let $\seqn{\mu_n}$ be a weak* null sequence of finitely supported measures on $St(\aA)$ for which there exists $M\io$ such that $\big\|\mu_n\big\|=1$ and $\big|\supp\big(\mu_n\big)\big|\le M$ for every $n\io$. If there is a subsequence $\seqk{\mu_{n_k}}$ such that $\big|\supp\big(\mu_{n_k}\big)\big|=1$ for every $k\io$, then every $\mu_{n_k}$ is simply of the form $\alpha_k\cdot\delta_{x_k}$ for some $\alpha_k\in\{-1,1\}$ and $x_k\in St(\aA)$. Consequently, for the constant unit function $1\in C(St(\aA))$ we have $\big|\mu_{n_k}(1)\big|=\big|\alpha_k\big|=1$ for every $k\io$, which contradicts the fact that $\seqn{\mu_n}$ converges weak* to $0$. It follows that for almost all $n\io$ we have $\big|\supp\big(\mu_n\big)\big|\ge2$ and so $M\ge 2$.

\medskip

We now prove the second part of the lemma. Let $m\io$ be the minimal number such that there exists a weak* null sequence $\seqn{\mu_n}$ of finitely supported measures on $St(\aA)$ such that $\big\|\mu_n\big\|=1$ and $\big|\supp\big(\mu_n\big)\big|=m$ for every $n\io$. By the previous paragraph, $m\ge2$. We will prove that in fact $m=2$.

First note that if there are a clopen set $U\sub St(\aA)$ and an increasing sequence $\seqk{n_k}$ such that $\mu_{n_k}\rstr U\neq 0$ and $\mu_{n_k}\rstr(St(\aA)\sm U)\neq 0$ for every $k\io$, then there is an increasing sequence $\seql{k_l}$ such that at least one of the sequences $\seql{\mu_l^1}$ and $\seql{\mu_l^2}$, defined for every $l\io$ as
\[\mu_l^1=\big(\mu_{n_{k_l}}\rstr(St(\aA)\sm U)\big)\big/\big\|\mu_{n_{k_l}}\rstr(St(\aA)\sm U)\big\|\]
and
\[\mu_l^2=\big(\mu_{n_{k_l}}\rstr U\big)\big/\big\|\mu_{n_{k_l}}\rstr U\big\|,\]
is weak* null. Since for every $l\io$ and $i\in\{1,2\}$ it holds $\big\|\mu_l^i\big\|=1$ and $\big|\supp\big(\mu_l^i\big)\big|<m$, we get a contradiction with the minimality of $m$. It follows that for every clopen $U\sub St(\aA)$ and almost all $n\io$ we have either $\supp\big(\mu_n\big)\sub U$, or $\supp\big(\mu_n\big)\cap U=\emptyset$. 

For every $n\io$ pick two distinct points $x_n,y_n\in\supp\big(\mu_n\big)$ and define the measure $\nu_n$ simply as follows $\nu_n=\frac{1}{2}\big(\delta_{x_n}-\delta_{y_n}\big)$. Of course, $\big\|\nu_n\big\|=1$. To finish the proof, we only need to show that $\seqn{\nu_n}$ is weak* null. But this is trivial, since for every clopen subset $U$ of $St(\aA)$ and almost all $n\io$ we have either $x_n,y_n\in U$, or $x_n,y_n\not\in U$; in either case it holds $\nu_n(U)=0$, so $\seqn{\nu_n}$ is pointwise null. Since $\seqn{\nu_n}$ is also uniformly bounded, by Corollary \ref{cor:pnub_wsn} it is weak* null (and so, by the minimality of $m$, we also have $m=2$).
\end{proof}

\begin{remark}\label{remark:reduction_m_2}
For a Boolean algebra $\aA$, if there exists a weak* null sequence $\seqn{\mu_n}$ of measures on $St(\aA)$ such that $\big\|\mu_n\big\|=1$ and $\big|\supp\big(\mu_n\big)\big|=2$ for every $n\io$, then one can also easily get such a sequence but with pairwise disjoint supports. Thus, from the above proof it basically follows that for every Boolean algebra $\aA$ the following two conditions are equivalent:
\begin{enumerate}
	\item there are a weak* null sequence $\seqn{\mu_n}$ of finitely supported measures on $St(\aA)$ and $M\io$ such that $\big\|\mu_n\big\|=1$ and $\big|\supp\big(\mu_n\big)\big|\le M$ for every $n\io$,
	\item there are two disjoint sequences $\seqn{x_n}$ and $\seqn{y_n}$ of distinct points in $St(\aA)$ such that for every clopen set $U$ and almost all $n\io$ we have: $x_n\in U$ if and only if $y_n\in U$.
\end{enumerate}
\end{remark}

The following lemma was the main tool used by Dow and Fremlin to obtain their result. We will need it, too.

\begin{lemma}[Dow--Fremlin {\cite[Lemma 2.2]{DF07}}]\label{lemma:df_lem}
Let $\aA\in V$ be a Boolean algebra. Assume that $\seqn{\dot{x}_n}$ is a sequence of $\B(\kappa)$-names for distinct ultrafilters on $\aA$. Let $G$ be a $\B(\kappa)$-generic filter over $V$. Then, for every condition $q\in\B(\kappa)$ there are a condition $p\le q$ and a sequence $\seqn{A_n}\in V$ of pairwise disjoint elements of $\aA$ such that $p\forces \clopen{A_n}_\aA\cap\big\{\dot{x}_k\colon k\io\big\}\neq\emptyset$ for every $n\io$.\noproof
\end{lemma}

In \cite{BNS21} Borodulin-Nadzieja and the first author proved that, for every $\B(\kappa)$-names $\dot{u}$ and $\dot{v}$ for ultrafilters on a given ground model Boolean algebra $\aA$, if $\forces_{\B(\kappa)}\dot{u}\neq\dot{v}$, then for every $\eps>0$ there are a condition $p\in\B(\kappa)$ and an element $C\in\aA$ such that $\mu_\kappa(p)>1/4-\eps$ and $p\forces C\in\dot{u}\triangle\dot{v}$. By exactly the same proof, \textit{mutatis mutandis}, we obtain the following formally stronger result.

\begin{lemma}[Cf. {\cite[Section 6.2]{BNS21}}]\label{lemma:bns}
Let $\aA\in V$ be a Boolean algebra. Let $\dot{u}$ and $\dot{v}$ be $\B(\kappa)$-names for ultrafilters on $\aA$. For every condition $q\in\B(\kappa)$, if $q\forces\dot{u}\neq\dot{v}$, then for every $\eps>0$ there are a condition $p\le q$ and an element $C\in\aA$ such that $\mu_\kappa(p)>\mu_\kappa(q)/4-\eps$ and $p\forces C\in\dot{u}\triangle\dot{v}$.\noproof
\end{lemma}

The next lemma is folklore.

\begin{lemma}\label{lemma:oobounding_ad}
Let $\PP$ be an $\oo$-bounding forcing notion and $G$ a $\PP$-generic filter over $V$. Let $X\in\cso\cap V[G]$. Then, in $V$, there is an uncountable almost disjoint family $\hH\sub\cso$ such that for every $Z\in\hH$ the intersection $X\cap Z$ is infinite.
\end{lemma}
\begin{proof}
In $V[G]$ let $\seqn{x_n}$ be the strictly increasing enumeration of elements of $X$. Since $\PP$ is $\oo$-bounding, there is a strictly increasing function $g\in\oo\cap V$ such that for every $k\io$ there is $n_k\io$ for which the following inequalities are satisfied:
\[\tag{$*$}g(2k)\le x_{n_k}\le g(2k+1)\]
(cf. the proof of \cite[Corollary 2.5]{DSNik}). In $V$, let $\fF\sub\cso$ be an uncountable almost disjoint family. For every $Y\in\fF$ set:
\[Z_Y=\bigcup_{n\in Y}\big\{g(2k),g(2k)+1,\ldots,g(2k+1)\big\}.\]
Put $\hH=\big\{Z_Y\colon Y\in\fF\big\}$. Of course, for every $Y\neq Y'\in\fF$ the intersection $Z_Y\cap Z_{Y'}$ is finite, so $\hH$ is an uncountable almost disjoint family. Also, by ($*$), for every $Y\in\fF$ the intersection $X\cap Z_Y$ is infinite.
\end{proof}

We are in the position to prove the main result of this section.

\begin{proof}[Proof of Theorem \ref{theorem:df_gen}]
Assume towards the contradiction that, in $V[G]$, there are a weak* null sequence $\seqn{\mu_n}$ of finitely supported measures on $St(\aA)$ and a number $M\io$ such that $\big\|\mu_n\big\|=1$ and $\big|\supp\big(\mu_n\big)\big|\le M$ for every $n\io$. By Remark \ref{remark:reduction_m_2}, 
we may assume that, in $V$, there are two sequences $\seqi{\dot{x}_i}$ and $\seqi{\dot{y}_i}$ of $\B(\kappa)$-names for ultrafilters on $\aA$ and a condition $q\in\B(\kappa)$ forcing that:
\begin{enumerate}[(a)]
	\item $\dot{x}_i\neq\dot{x}_j$ and $\dot{y}_i\neq\dot{y}_j$ for every $i\neq j\io$,
	\item $\dot{x}_i\neq\dot{y}_j$ for every $i,j\io$,
	\item for every $A\in\aA$ and almost all $i\io$ we have: $\dot{x}_i\in\clopen{A}_\aA$ if and only if $\dot{y}_i\in\clopen{A}_\aA$.
\end{enumerate}
By Lemma \ref{lemma:df_lem}, there are a condition $p\le q$ and a sequence $\seqn{A_n}$ (in $V$!) of pairwise disjoint elements of $\aA$ such that for every $n\io$ we have $p\forces \clopen{A_n}_\aA\cap\big\{\dot{x}_i\colon\ i\io\big\}\neq\emptyset$. For each $n\io$ by $\dot{F}_n$ denote a $\B(\kappa)$-name such that $p$ forces that
\[\dot{F}_n=\big\{i\io\colon\ \big|\clopen{A_n}_\aA\cap\big\{\dot{x}_i,\dot{y}_i\big\}\big|=1\big\}.\]
Note that, by property (c), $p$ forces that each $\dot{F}_n$ is at most finite.

\medskip

We need to consider two cases. 

\medskip

(1) $p$ forces that for almost all $n\io$ the set $\dot{F}_n$ is empty. Let $r\le p$ and $N\io$ be such that for every $n\ge N$ we have $r\forces\dot{F}_n=\emptyset$. The last formula simply means that for every $n\ge N$ and $i\io$ the following holds:
\[r\forces\dot{x}_i\in\clopen{A_n}_\aA\Leftrightarrow\dot{y}_i\in\clopen{A_n}_\aA.\]
For each $n\ge N$, let $\dot{\alpha}_n$ be a $\B(\kappa)$-name such that
\[r\forces\dot{\alpha}_n=\min\big\{i\io\colon\ \dot{x}_i\in\clopen{A_n}_\aA\big\},\]
and let $\dot{u}_n$ and $\dot{v}_n$ be $\B(\kappa)$-names for ultrafilters on $\aA$ such that
\[r\forces\dot{u}_n=\dot{x}_{\dot{\alpha}_n}\text{ and }\dot{v}_n=\dot{y}_{\dot{\alpha}_n}.\]
Of course, for $n\neq m\ge N$ we have $r\forces\dot{\alpha}_n\neq\dot{\alpha}_m$.

For every $n\ge N$, since $r\forces\dot{u}_n\neq\dot{v}_n$, by Lemma \ref{lemma:bns}, we get a condition $r_n\le r$ and an element $C_n\in\aA$ such that $\mu_\kappa\big(r_n\big)>\mu_\kappa(r)/5$ and $r_n\forces C_n\in\dot{u}_n\triangle\dot{v}_n$. Since $r_n\forces \dot{u}_n,\dot{v}_n\in\clopen{A_n}_\aA$, we may actually assume that $C_n\le A_n$. In particular, $C_n\wedge C_m=0$ for every $n\neq m\ge N$. Set:
\[s=\bigwedge_{n\ge N}\bigvee_{k\ge n}r_k;\]
then, $\mu_\kappa(s)\ge\mu_\kappa(r)/5$, so $s\in\B(\kappa)$ and $s\le r$. It follows that $s$ forces that $C_n\in\dot{u}_n\triangle\dot{v}_n$ for infinitely many $n\ge N$, or, in other words, that 
\[\tag{$+$}\big|\clopen{C_n}_\aA\cap\big\{\dot{u}_n,\dot{v}_n\big\}\big|=1\]
for infinitely many $n\ge N$. Since the sequence $\seq{C_n}{n\ge N}$ is in $V$, its supremum exists in $\aA$, so set
\[C=\bigvee_{n\ge N}C_n,\]
and note that for every $n\ge N$ we have $C\wedge A_n=C_n$ (since $\seqn{A_n}$ is pairwise disjoint).

We claim that $s$ forces that $\big|\clopen{C}_\aA\cap\big\{\dot{x}_i,\dot{y}_i\big\}\big|=1$ for infinitely many $i\io$, contradicting condition (c). Indeed, observe that for every $n\ge N$ we have:
\begin{align*}
s\forces\clopen{C}_\aA\cap\big\{\dot{x}_{\dot{\alpha}_n},\dot{y}_{\dot{\alpha}_n}\big\}&\sub\clopen{C}_\aA\cap\clopen{A_n}_\aA\cap\big\{\dot{x}_{\dot{\alpha}_n},\dot{y}_{\dot{\alpha}_n}\big\}=\\
&=\clopen{C\wedge A_n}_\aA\cap\big\{\dot{x}_{\dot{\alpha}_n},\dot{y}_{\dot{\alpha}_n}\big\}=\clopen{C_n}_\aA\cap\big\{\dot{x}_{\dot{\alpha}_n},\dot{y}_{\dot{\alpha}_n}\big\},
\end{align*}
and so, since $C_n\le C$,
\[s\forces\clopen{C}_\aA\cap\big\{\dot{x}_{\dot{\alpha}_n},\dot{y}_{\dot{\alpha}_n}\big\}=\clopen{C_n}_\aA\cap\big\{\dot{x}_{\dot{\alpha}_n},\dot{y}_{\dot{\alpha}_n}\big\}.\]
By ($+$), it follows that
\[s\forces\big|\clopen{C}_\aA\cap\big\{\dot{x}_{\dot{\alpha}_n},\dot{y}_{\dot{\alpha}_n}\big\}\big|=\big|\clopen{C_n}_\aA\cap\big\{\dot{x}_{\dot{\alpha}_n},\dot{y}_{\dot{\alpha}_n}\big\}\big|=1\]
for infinitely many $n\ge N$.
\medskip

(2) There is a condition $r\le p$ forcing that for infinitely many $n\io$ the set $\dot{F}_n$ is non-empty. Let then $\dot{f}$ be a $\B(\kappa)$-name for a function $\omega\to\omega$ such that for every $n\io$ we have:
\[
r\forces\dot{f}(n)=
\begin{cases}
	\max\big\{m\ge n\colon\ \dot{F}_n\cap\dot{F}_m\neq\emptyset\big\},&\text{ if }\dot{F}_n\neq\emptyset,\\
	n,&\text{ if }\dot{F}_n=\emptyset.
\end{cases}
\]
Since $r$ forces that $\dot{F}_n$ is finite for every $n\io$ and that each $i\io$ belongs to at most two different $\dot{F}_n$'s, the above definition is valid. For every $n,m\io$ we also have $r\forces\dot{f}(n)\ge n$ and
\[\tag{$*$}r\forces\big(m>\dot{f}(n)\Rightarrow\dot{F}_n\cap\dot{F}_m=\emptyset\big).\]

Since $\B(\kappa)$ is $\omega^\omega$-bounding, there are a strictly increasing function $g\colon\omega\to\omega$ such that $g(0)>0$ and a condition $s\le r$ such that $s\forces \dot{f}(n)<g(n)$ for every $n\io$. Define the function $h\colon\omega\to\omega$ for every $n\io$ as: $h(n)=g^{n+1}(0)=(g\circ g\circ\ldots\circ g)(0)$, where the composition is taken $n+1$ times.

\medskip

We have two subcases of case (2).

\medskip

(2a) There is a condition $t\le s$ forcing that for infinitely many $n\io$ the set $\dot{F}_{h(n)}$ is non-empty. Let $H$ be a $\B(\kappa)$-generic filter over $V$ containing $t$. For every $\B(\kappa)$-name $\sigma$ by $\sigma^H$ we denote its evaluation in $V[H]$. 

We now work in $V[H]$. For every $n<m\io$ we have:
\[h(m)=g^{m+1}(0)=g^{m-n}\big(g^{n+1}(0)\big)=g^{m-n}(h(n))\ge g(h(n))>\dot{f}^H(h(n)),\]
so, by ($*$), it holds that
\[\tag{$**$}\dot{F}_{h(n)}^H\cap\dot{F}_{h(m)}^H=\emptyset.\]
Since $\aA$ is $\sigma$-complete in $V$ and $h\in V$, the supremum
\[A=\bigvee_{n\io}A_{h(n)}\]
exists in $\aA$. Let $X\in\cso$ be any set such that for every $n\in X$ it holds $\dot{F}_{h(n)}^H\neq\emptyset$, so we may pick $i_n\in \dot{F}_{h(n)}^H$. Note that, by condition (c), there is no $X'\in\ctblsub{X}$ such that $\big|\big\{\dot{x}_{i_n}^H,\dot{y}_{i_n}^H\big\}\cap\clopen{A}_\aA\big|=1$ for every $n\in X'$. It follows that---removing a finite number of elements of $X$ if necessary---for every $n\in X$ we have $\big\{\dot{x}_{i_n}^H,\dot{y}_{i_n}^H\big\}\sub\clopen{A}_\aA$. Property ($**$) implies that for every $n\in X$ we have:
\[\Big|\big\{\dot{x}_{i_n}^H,\dot{y}_{i_n}^H\big\}\cap\big(\clopen{A}_\aA\sm\bigcup_{n\io}\clopen{A_{h(n)}}_\aA\big)\Big|=1,\]
that is, there is a unique point in $\big\{\dot{x}_{i_n}^H,\dot{y}_{i_n}^H\big\}$ which belongs to the boundary (in $St(\aA)$) of the open set $\bigcup_{n\io}\clopen{A_{h(n)}}_\aA$.

By Lemma \ref{lemma:oobounding_ad}, there is an uncountable almost disjoint family $\hH\sub\cso$ in $V$ such that the intersection $X\cap Z$ is infinite for every $Z\in\hH$. Since every $Z\in\hH$ is in $V$, the supremum
\[A_Z=\bigvee_{n\in Z}A_{h(n)}\]
exists in $\aA$. For every $Z\in\hH$ set:
\[B_Z=\clopen{A_Z}_\aA\sm\bigcup_{n\in Z}\clopen{A_{h(n)}}_\aA.\]
For every $Z_1\neq Z_2\in\hH$ we have:
\[\clopen{A_{Z_1}}_\aA\cap\clopen{A_{Z_2}}_\aA=\bigcup_{n\in Z_1\cap Z_2}\clopen{A_{h(n)}}_\aA,\]
and hence $B_{Z_1}\cap B_{Z_2}=\emptyset$. Since $\hH$ is uncountable and $X$ is countable, it follows that there is $Z\in\hH$ such that $\big\{\dot{x}_{i_n}^H,\dot{y}_{i_n}^H\big\}\cap B_Z=\emptyset$ for every $n\in X\cap Z$, and hence $\big|\big\{\dot{x}_{i_n}^H,\dot{y}_{i_n}^H\big\}\cap\clopen{A_Z}_\aA\big|=1$ for infinitely many $n\in X$, which contradicts condition (c).

\medskip

(2b) $s$ forces that for almost all $n\io$ the set $\dot{F}_{h(n)}$ is empty. We proceed exactly in the same way as in case (1), that is, using Lemma \ref{lemma:bns} we obtain $N\io$, an antichain $\seq{C_n}{n\ge N}$ in $V$ such that $C_n\le A_{h(n)}$ for every $n\ge N$, and a condition $t\le s$ forcing that $\big|\clopen{\bigvee_{n\io}C_n}_\aA\cap\big\{\dot{x}_i,\dot{y}_i\big\}\big|=1$ for infinitely many $i\io$, which again contradicts condition (c).
\end{proof}

\subsection{Dominating reals\label{sec:dominating}}

Let $\PP\in V$ be a notion of forcing and $G$ a $\PP$-generic filter over $V$. Recall that a real $f\ioo\cap V[G]$ is \textit{dominating over $V$} if $g\le^*f$ for every $g\ioo\cap V$. By $\ooi$ we denote the family of all those functions $f\ioo$ which are increasing, that is, $f(n)\le f(n+1)$ for every $n\io$, and $\lim_{n\to\infty}f(n)=\infty$. Let us then also say that a real $h\in(\oo)^\infty\cap V[G]$ is \textit{anti-dominating over $V$} if $h\le^*g$ for every $g\iooi\cap V$.

It appears that adding a dominating real is equivalent to adding an anti-dominating real. To prove it, we need to introduce the following auxiliary operator $\Phi\colon\ooi\to\ooi$. It seems that the idea standing behind $\Phi$, and hence also behind Propositions \ref{prop:phi_props} and \ref{prop:domin_anti_domin}, is standard (cf. e.g. Canjar \cite[Sections 1.5 and 3.6]{Can88}).

Let $f\iooi$ and write $\ran(f)=\big\{n_1^f<n_2^f<n_3^f<\ldots\big\}$. Set also $n_0^f=-1$, so always $n_0^f<n_1^f$. Note that for every $n\io$ there is  unique $i\io$ such that $n_i^f\le n<n_{i+1}^f$. Put:
\[\Phi(f)(n)=\min f^{-1}\big(n_{i+1}^f\big).\]
It is immediate that $\Phi(f)\iooi$. Note that $\Phi(\id_\omega)(n)=n+1$ for every $n\io$. The next proposition lists most basic properties of $\Phi$.

\begin{proposition}\label{prop:phi_props}
For every $f,g\iooi$ the following conditions hold:
\begin{enumerate}
	\item $\big(f\circ\Phi(f)\big)>\id_\omega$,
	\item $\big(\Phi(f)\circ f\big)>\id_\omega$,
	\item $\Phi(\Phi(f))=f$,
	\item if $f\le^*g$, then $\Phi(g)\le^*\Phi(f)$.
\end{enumerate}
\end{proposition}
\begin{proof}
Let $f,g\iooi$. Enumerate:
\[\ran(f)=\big\{n_1^f<n_2^f<n_3^f<\ldots\big\}\quad\text{and}\quad\ran(g)=\big\{n_1^g<n_2^g<n_3^g<\ldots\big\},\]
and set $n_0^f=-1$ and $n_0^g=-1$.

We first prove (1) and (2). Fix $n\io$ and let $i,j\io$ be such that $n_i^f\le n< n_{i+1}^f$ and $n_j^f=f(n)$. We have:
\[\big(f\circ\Phi(f)\big)(n)=f\big(\min f^{-1}\big(n_{i+1}^f\big)\big)=n_{i+1}^f>n,\]
which proves (1). To see (2), note that the monotonicity of $f$ implies that $\min f^{-1}\big(n_{j+1}^f\big)>n$ and thus we have:
\[\big(\Phi(f)\circ f\big)(n)=\Phi(f)(f(n))=\min f^{-1}\big(n_{j+1}^f\big)>n.\]

Let us now prove (3). For every $i\ge 1$ set $n_i'=\min f^{-1}\big(n_i^f\big)$. By the monotonicity of $f$, $n_{i+1}'>n_i'$ for every $i\ge 1$. 
Note that $n_1'=0$ and
\[(\Phi(f))^{-1}\big(n_{1}'\big)=\big\{0,1,2,\ldots,n_{1}^f-1\big\},\]
so $(\Phi(f))^{-1}\big(n_{1}'\big)=\emptyset$ if $n_1^f=0$, as well as for each $i\ge 1$ we have:
\[(\Phi(f))^{-1}\big(n_{i+1}'\big)=\big\{n_i^f,n_i^f+1,n_i^f+2,\ldots,n_{i+1}^f-1\big\}.\]
Fix $n\io$ and let $i\ge 1$ be such that $n_i'\le n<n_{i+1}'$. It means that
\[\min f^{-1}\big(n_i^f\big)\le n<\min f^{-1}\big(n_{i+1}^f\big),\]
so $f(n)=n_i^f$. It holds:
\[(\Phi(\Phi(f))(n)=\min\Big((\Phi(f))^{-1}\big(n_{i+1}'\big)\Big)=n_i^f=f(n),\]
which implies (3).

Finally, we shall prove (4). Assume that $f\le^*g$. There exists $N\io$ such that $f(n)\le g(n)$ for every $n\ge N$. Let $n>f(N)$. There are $i,j,k\io$ such that $n_i^f\le n<n_{i+1}^f$, $n_j^g\le n<n_{j+1}^g$, and $n_k^f=f(N)$. Note that $k\le i$, so $n_k^f\le n_i^f$. Set:
\[l=\Phi(f)(n)=\min f^{-1}\big(n_{i+1}^f\big)\]
and
\[m=\Phi(g)(n)=\min g^{-1}\big(n_{j+1}^g\big).\]
We claim that $m\le l$, so for the sake of contradiction let us assume that $l<m$. We then have:
\[g(l)\le n_j^g\le n<n_{i+1}^f=f(l),\]
so $g(l)<f(l)$. But since $f$ is increasing, it holds:
\[l=\min f^{-1}\big(n_{i+1}^f\big)>\max f^{-1}\big(n_i^f\big)\ge\max f^{-1}\big(n_k^f\big)\ge N,\]
so $f(l)\le g(l)$, which is a contradiction.
\end{proof}

\begin{proposition}\label{prop:domin_anti_domin}
Let $\PP\in V$ be a notion of forcing. Then, $\PP$ adds a dominating real if and only if it adds an anti-dominating real.
\end{proposition}
\begin{proof}
Let $G$ be a $\PP$-generic filter over $V$. We work in $V[G]$. Assume that there is a dominating real $f\ioo$ over $V$ and define an auxiliary function $g\ioo$ as follows:
\[g(n)=n+\max\{f(m)\colon m\le n\},\]
where $n\io$. Obviously, $g\iooi$ and it is also a dominating real over $V$, so for every $h\iooi\cap V$ we have $h\le^* g$.

For every $h\iooi\cap V$, we have $\Phi(h)\iooi\cap V$ and, by Proposition \ref{prop:phi_props}.(3), $h=\Phi(\Phi(h))$. It follows that
\[\tag{$*$}\Phi\big[\ooi\cap V]=\ooi\cap V.\]
Since $g\ioo$ and $g$ is dominating every $h\iooi\cap V$, we get by ($*$) and Proposition \ref{prop:phi_props}.(4) that $\Phi(g)\le^*h$ for every $h\iooi\cap V$. In other words, we get that $\Phi(g)$ is an anti-dominating real over $V$. 

The proof in the other direction is similar.
\end{proof}

We are ready to prove the main result of this section.

\dominating*
\begin{proof}
We first work in $V$. By the Josefson--Nissenzweig theorem (see \cite[Chapter XII]{Die84}) and the Riesz representation theorem, there is a weak* null sequence $\seqn{\mu_n}$ of measures on the Boolean algebra $\aA$ such that $\big\|\mu_n\big\|=1$ for every $n\io$. For every $A\in\aA$ define the sequences $c_A,d_A\in\R^\omega$ as follows:
\[c_A(n)=\min\Big\{\big|\mu_n(A)\big|+1/n,\ 1\Big\},\]
\[d_A(n)=\min\big\{1/m\colon\ m\io,\ c_A(k)\le 1/m\text{ for all }k\ge n\big\},\]
where $n\io$. Then, $c_A(n)>0$ and 
\[0\le\big|\mu_n(A)\big|\le c_A(n)\le d_A(n)\le 1\]
for every $n\io$, and
\[\lim_{n\to\infty}d_A(n)=\lim_{n\to\infty}c_A(n)=0.\]
Finally, for every $A\iA$ and $n\io$ set $e_A(n)=1/d_A(n)$. It follows that $e_A\iooi$.

Let us now go to $V[G]$. $\PP$ adds a dominating real, so by Proposition \ref{prop:domin_anti_domin} there is an anti-dominating real $g\iooi\cap V[G]$ over $V$. By taking the function $\max(g,1)$ instead of $g$, we may assume that $g(n)>0$ for every $n\io$. For every $A\iA$ we have $g\le^*e_A$, so if  we define the sequence $c\in\R^\omega$ by the formula $c(n)=1/g(n)$, where $n\io$, then we get that $d_A\le^* c$ for every $A\iA$. Of course, $c(n)>0$ for every $n\io$ and $\lim_{n\to\infty}c(n)=0$.

For every $n\io$ define the measure $\nu_n$ on $\aA$ as follows:
\[\nu_n(A)=\mu_n(A)/c(n),\]
where $A\iA$. Note that $\big\|\mu_n\big\|=1$ yields that
\[\big\|\nu_n\big\|=\big\|\mu_n\big\|/c(n)=1/c(n)=g(n),\]
so $\sup_{n\io}\big\|\nu_n\big\|=\infty$, as $g\iooi$. On the other hand, for every $A\iA$ we have
\[\big|\nu_n(A)\big|=\big|\mu_n(A)\big|/c(n)\le d_A(n)/c(n)\le 1\]
for sufficiently large $n\io$, so $\sup_{n\io}\big|\nu_n(A)\big|<\infty$ for every $A\iA$. It follows that the sequence $\seqn{\nu_n}$ is pointwise bounded but not uniformly bounded, hence, by Lemma \ref{lemma:nik_equiv_forms}, $\aA$ does not have the Nikodym property in $V[G]$.
\end{proof}

%
%
%
%
%
%
%
%

\section{Cardinal characteristics of the continuum\label{sec:cardinal}}

In this section we provide several consequences of Theorem \ref{theorem:random} to cardinal characteristics of the continuum. For basic information concerning various standard cardinal characteristics, we refer the reader to Blass \cite{BlaHBK}.

We start with the definitions of two characteristics $\fraknik$ and $\frakgr$ which we call \textit{the Nikodym number} and \textit{the Grothendieck number}, respectively:
\[\fraknik=\min\big\{|\aA|\colon\ \aA\text{ is an infinite Boolean algebra with the Nikodym property}\big\},\]
and
\[\frakgr=\min\big\{|\aA|\colon\ \aA\text{ is an infinite Boolean algebra with the Grothendieck property}\big\}.\]
A detailed discussion on the estimations of $\fraknik$ and $\frakgr$ in terms of standard cardinal characteristics of the continuum occurring in Cicho\'n's and van Douwen's diagrams as well as on miscellaneous consistency results one can find in the survey paper \cite{DSsurv}. In \cite{SZForExt} the authors proved that in the Miller model the inequality $\fraknik=\frakgr<\frakd$ holds. We now show that the proof of Theorem \ref{theorem:random} easily implies that the converse inequality may also consistently hold.

Similarly as in Section \ref{sec:df}, for an infinite set $I$ let $\mu_I$ denote the standard product probability measure on the space $2^I$ and let $\B(I)=Bor\big(2^I\big)\big/\big\{A\in Bor\big(2^I\big)\colon\ \mu_I(A)=0\big\}$ be its measure algebra. Again, it is well-known that $\B(I)$ is a $\oo$-bounding poset adding $|I|$ many random reals (see \cite[Section 3.1]{BJ95}).

\begin{corollary}\label{cor:nik_gr_d}
Let $\kappa$ be an infinite cardinal number. Let $G$ be a $\B(\kappa)$-generic filter over $V$. Then, in $V[G]$, there is no infinite Boolean algebra of size $<\kappa$ with the Nikodym property or the Grothendieck property. 

Consequently, in the random model every infinite Boolean algebra of size $\le\frakd$ has neither the Nikodym property nor the Grothendieck property.
\end{corollary}
\begin{proof}
We work in $V[G]$. Let $\aA$ be an infinite Boolean algebra of size $<\kappa$ and $F=\big\{x_n\colon\ n\io\big\}$ be a countable subset of its Stone space such that $x_n\neq x_m$ for $n\neq m\io$. By the standard argument based on $\B(\kappa)$ being c.c.c., there is $I\subset\kappa$ such that $|I|=|\aA|<\kappa$ and
\[\big\{[A]_\aA\cap F\colon\ A\in\aA\big\}\in V[G\rstr I].\]
In $V[G]$ there is a random real $r\in\Cantor$ over $V[G\rstr I]$. Now it suffices to consider the sequence $\seqn{\mu_n}$ of measures on $St(\aA)$, defined for every $n\io$ by the formula:
\[\nu_n=\alpha_n\cdot\sum_{i\in I_n}(-1)^{r(i)+1}\delta_{x_i},\]
where $\alpha_n$ and $I_n$ are as previously, and repeat the proof of Theorem \ref{theorem:random}.
\end{proof}

\begin{remark}
Note that the same argument as in the above proof works, e.g., for finite support iterations of $\B(\omega)$ of length $\kappa$ for regular uncountable $\kappa$.
\end{remark}

\begin{corollary}\label{cor:d_nik_gr}
$\frakd<\fraknik=\frakgr$ holds in the random model. \noproof
\end{corollary}

Corollary \ref{cor:d_nik_gr}, together with the aforementioned fact that in the Miller model we have $\frakd>\fraknik=\frakgr$, yields the following independence result.

\begin{corollary}
Let $\frakx\in\{\fraknik,\frakgr\}$. Neither of the inequalities $\frakx\le\frakd$ and $\frakx\ge\frakd$ is provable in ZFC.\noproof
\end{corollary}

A close relative to the numbers $\fraknik$ and $\frakgr$ is \textit{the convergence number} $\frakz$ defined as follows:
\begin{align*}
\frakz=\min\big\{w(K)\colon\ K\text{ is an infinite }&\text{compact space}\\&\text{with no non-trivial convergent sequences}\big\}.
\end{align*}
Here $w(K)$ denotes the weight of $K$. The number $\frakz$ was studied e.g. in Brian and Dow \cite{BD19}. It is immediate that $\frakz\le\fraknik$ and $\frakz\le\frakgr$. By the result of Dow and Fremlin \cite{DF07} stating that in any random extension $V[G]$, for every $\sigma$-complete Boolean algebra $\aA\in V$, its Stone space $St(\aA)$ does not contain any non-trivial convergent sequences, we have that $\frakz=\omega_1<\frakc$ in the random model. Thus, by Corollary \ref{cor:d_nik_gr}, we immediately get also the following fact.

\begin{corollary}
$\omega_1=\frakz<\fraknik=\frakgr=\frakc$ holds in the random model.\noproof
\end{corollary}

Dow \cite{Dow21} proved that in the Laver model there are (totally disconnected) compact spaces of weight $\omega_1$ containing no non-trivial convergent sequences, so $\frakz=\omega_1$ holds in this model. On the other hand, it is well known that the bounding number $\frakb$ has value $\omega_2$ in the Laver model, and it was proved by the first author in \cite[Proposition 3.2]{DSNik} that $\frakb\le\fraknik$ holds in ZFC. We thus get the following corollary.

\begin{corollary}\label{cor:z_nik_laver}
$\omega_1=\frakz<\fraknik=\omega_2$ holds in the Laver model.\noproof
\end{corollary}

We do not know the value of $\frakgr$ in the Laver model (cf. Question \ref{ques:laver_gr}).

\section{Open questions\label{sec:open}}

\subsection{Dominating reals and the Grothendieck property}

In the introductory section we admitted that, contrary to the case of the Nikodym property, we do not know whether adding dominating reals kills the Grothendieck property of ground model $\sigma$-complete Boolean algebras.

\begin{question}\label{ques:laver_gr}
Let $\aA\in V$ be an infinite $\sigma$-complete Boolean algebra. Assume that $G$ is a generic filter for the Laver forcing over $V$. Does $\aA$ have the Grothendieck property in $V[G]$?
\end{question}

\begin{question}\label{ques:dom_gr}
Does there exist a notion of forcing $\PP$ adding dominating reals and such that in any $\PP$-generic extension $V[G]$ any ground model $\sigma$-complete Boolean algebra has the Grothendieck property?
\end{question}

An affirmative answer to Question \ref{ques:laver_gr} would yield a new consistent example of a Boolean algebra with the Grothendieck property but without the Nikodym property. Recall that while there are many consistent or even ZFC examples of Boolean algebras with the Nikodym property but without the Grothendieck property, see e.g. \cite{Sch82}, \cite{GW83}, \cite{SZMinGen}, so far only one example of an algebra with the Grothendieck property and without the Nikodym property has been found---the construction was obtained by Talagrand \cite{Tal84} under the assumption of the Continuum Hypothesis. 

\subsection{Eventually different reals}

Let $V[G]$ be a $\PP$-generic extension of the ground model $V$ for some forcing notion $\PP$. If $f\in\oo\cap V[G]$ is a dominating real, then obviously it is \textit{an eventually different real}, that is, for every $g\in\oo\cap V$ the set $\big\{n\io\colon\ f(n)=g(n)\big\}$ is finite. The converse does not hold, as e.g. the random forcing or the eventually different forcing both add eventually different reals but not dominating reals. Since the latter forcing adds Cohen reals, too, by Theorems \ref{theorem:cohen} and \ref{theorem:random} both notions kill the Nikodym and Grothendieck properties of infinite ground model Boolean algebras. Thus, it seems that all the standard classical notions adding eventually different reals kill at least one of the properties. It is also a folklore fact that a forcing adds an eventually different real if and only if it makes the ground model reals meager, hence, trivially by the assumption, the notions of forcing considered in \cite{SZForExt} (cf. the third paragraph of Introduction), which are proved therein to preserve both the Nikodym property and the Grothendieck property of ground model $\sigma$-complete Boolean algebras, do not add eventually different reals. So it seems reasonable to ask whether adding an eventually different real is solely a reason that ground model Boolean algebras lose their Nikodym property (and, in the view of Question \ref{ques:dom_gr}, possibly also the Grothendieck property).

\begin{question}\label{ques:ev_diff_nik}
Does there exist a notion of forcing $\PP$ adding eventually different reals and such that in any $\PP$-generic extension $V[G]$ any ground model $\sigma$-complete Boolean algebra has the Nikodym property?
\end{question}

\begin{question}\label{ques:ev_diff_gr}
Does there exist a notion of forcing $\PP$ adding eventually different reals and such that in any $\PP$-generic extension $V[G]$ any ground model $\sigma$-complete Boolean algebra has the Grothendieck property?
\end{question}

\subsection{Cardinal characteristics $\fraknik$ and $\frakgr$}

We are not aware of any model in which the numbers $\fraknik$ and $\frakgr$ have different values. Thus, we pose the following question.

\begin{question}\label{ques:nik_gr_diff}
Is it consistent that $\fraknik<\frakgr$ or $\fraknik>\frakgr$?
\end{question}

Note that an affirmative answer to Question \ref{ques:laver_gr} would imply that $\omega_1=\frakgr<\fraknik=\frakc$ holds in the Laver model (cf. Corollary \ref{cor:z_nik_laver}).

\end{document}